\documentclass[11pt,reqno]{amsart}

\usepackage{amsthm,amsmath,a4,amssymb,stmaryrd,enumerate,prooftree,boxedminipage}
\usepackage[all]{xy}

\newtheorem{theorem}{Theorem}
\newtheorem*{theorem*}{Theorem}
\newtheorem{proposition}[theorem]{Proposition}

\newtheorem{lemma}[theorem]{Lemma}

\theoremstyle{definition}

\theoremstyle{remark}
\newtheorem{remark}[theorem]{Remark}
\newtheorem*{example*}{Example}
\newtheorem*{remark*}{Remark}
\newtheorem*{notation*}{Notation}
\newtheorem*{convention*}{Convention}

\usepackage[all]{xy}
 \SelectTips{cm}{}
 \newdir{ >}{{}*!/-7pt/@{>}}

\newcommand{\myemph}{\textit}

\newcommand{\defeq}{=_\mathrm{def}}
\newcommand{\myin}{\in}

\newcommand{\cdl}[2][]{\xymatrix@1#1{#2}}

\newcommand{\map}[1]{\mathcal{#1}}
\newcommand{\mapA}{\map{A}}
\newcommand{\mapB}{\map{B}}

\newcommand{\mapM}{\map{M}}

\newcommand{\leftpitch}[1]{{#1}^\pitchfork}
\newcommand{\rightpitch}[1]{{}^\pitchfork  #1}

\newcommand{\Mpitch}{\leftpitch{\mapM}}
\newcommand{\pitchM}{\rightpitch{\mapM}}

\newcommand{\typefont}[1]{\mathrm{#1}}

\newcommand{\type}{\typefont{Type}}
\newcommand{\Id}{\typefont{Id}}
\newcommand{\refl}{\operatorname{\mathrm{r}}}
\newcommand{\myJ}{\typefont{J}}

\newcommand{\theory}[1]{\mathbb{#1}}
\newcommand{\theoryT}{\theory{T}}

\newcommand{\classarg}[1]{\mathcal{C}(#1)}
\newcommand{\classT}{\classarg{\theoryT}}

\newcommand{\Path}{\mathrm{Path}}

\newcommand{\catE}{\mathcal{E}}

\newcommand{\Gpd}{\mathbf{Gpd}}

\newcommand{\contype}{\mathrm{Cxt}}

\title{The identity type weak factorisation system}

\author[N. Gambino]{Nicola Gambino}

\address{Department of Computer Science, University of Leicester, University Road, LE1 7RH, England
and Centre de Recerca Matem\`atica, Apartat 50, E-08193 Bellaterra, Spain}

\email{nicola.gambino@gmail.com}

\author[R. Garner]{Richard Garner}

\address{Department of Mathematics, Uppsala University, Box 480, S-751 06, Uppsala, Sweden}

\email{rhgg2@cam.ac.uk}

\date{\today}

\begin{document}

\begin{abstract}
We show that the classifying category $\classT$ of a dependent type theory $\theoryT$ with axioms for identity types admits a non-trivial weak factorisation system. We provide an explicit characterisation of the elements of both the left class and the right class of the weak
factorisation system. This characterisation is applied to relate identity types and the homotopy theory of groupoids.
\end{abstract}

\maketitle

\section{Introduction}

From the point of view of mathematical logic and theoretical computer science, Martin-L\"of's  axioms for identity
types~\cite{NordstromB:marltt}  admit a  conceptually clear explanation in terms of the propositions-as-types correspondence~\cite{HowardW:foratn,MartinLofP:inttt,ScottD:conv}.
The fundamental idea behind this explanation is that, for any two elements $a, b$ of a type~$A$, we
have a new type $\Id_A(a,b)$, whose elements
are to be thought of as proofs that $a$ and $b$ are  equal.
Yet, identity types determine a highly complex structure on each type, which is far from being fully understood. A glimpse of this structure reveals itself as soon as we start applying the construction of identity types iteratively: not only do we have proofs of equality between two elements of a type, but also of proofs of equality between such proofs, and so on.
The difficulty of isolating the structure determined by identity types  is closely related to the problem of describing a satisfactory category-theoretic semantics
for them. For example,  the semantics arising from  locally cartesian closed
categories~\cite{HofmannM:intttl,SeelyR:locccc} validates not only the axioms for identity types, but also additional axioms, known as the reflection rules, which make identity types essentially trivial. To improve on this unsatisfactory situation and obtain models that do not validate the reflection rules, Awodey and  Warren have recently introduced a semantics of identity types in categories  equipped with a weak factorisation system~\cite{AwodeyS:homtmi}.

Our aim here is  to advance our understanding of the categorical structure implicit in the axioms for
identity types. We do so by providing further evidence of a close connection between the axioms for
identity types and  the notion of
a weak factorisation system. Our main result states that if~$\theoryT$ is a dependent type theory with the axioms for identity types, then its classifying category~$\classT$ admits a non-trivial weak factorisation system,  which we
shall refer to as the \myemph{identity type weak factorisation system}. This result should be regarded as
analogous to the fundamental result exhibiting the structure of a cartesian closed category on the
classifying category of  the simply-typed $\lambda$-calculus~\cite{LambekJ:inthoc,ScottD:reltlc}.
As such, it provides also a contribution to  the development of a functorial semantics~\cite{LawvereW:funs} for dependent type theories~\cite{AwodeyS:homtmi,CartmellJ:genatc,DybjerP:inttt,HofmannM:synsdt,JacobsB:catltt,MaiettiM:modcbd}.

A remarkable feature of the  identity type weak factorisation system is that its
 definition  does not involve the notion of identity types, but only a canonical class of maps
in $\classT$, generally referred to as \myemph{display maps}~\cite{JacobsB:catltt,TaylorP:prafm}. Indeed, the axioms for identity types are
used  only to verify that the appropriate axioms hold. After having established the existence of the identity type weak factorisation system, we will provide an explicit characterisation of its  classes of maps. This will lead us to two applications. The first establishes
a stability property of the identity type weak factorisation system; the second  provides
further insight into the relationship between dependent type theories with identity types and the category of groupoids, which we denote $\Gpd$.

The idea of relating identity types and groupoids dates back to the discovery of the groupoid model of type theory by Hofmann and Streicher~\cite{HofmannM:gromtt}, who showed that the axioms for  identity types allow us to construct a groupoid out of every type. We
develop this idea  in three directions. First, we generalise it by constructing a groupoid
out of every context, which we think of as a family of types, rather than out of a single type. Secondly,
we extend this construction to a functor~$\mathcal{F} : \classT \rightarrow \mathbf{Gpd}$. Finally, we use the functor $\mathcal{F}$
to relate the identity type weak factorisation system to the natural Quillen model structure on the category of
groupoids~\cite{AndersonD:fibgr,JoyalA:strscs}, by showing how the identity type weak factorisation system is mapped
into the weak factorisation system determined by injective equivalences and Grothendieck
fibrations in $\Gpd$.

\section{Identity types}
\label{sec:identitytypes}

The dependent type theories that we consider allow us to make judgements of four forms:
\begin{equation}
\label{equ:judge}
A \in \type \, ,  \quad
a \in A \, ,  \quad
A = B \in \type \, , \quad
a = b \in A \, .
\end{equation}
They assert, respectively,  that $A$ is a type, that $a$ is an element of $A$, that~$A$ and $B$
are definitionally equal types, and that $a$ and $b$ are definitionally equal elements of $A$.
We speak of \myemph{definitional equality}, rather than just equality, since the axioms for identity types will
provide us with a second notion of equality, which we will refer to as \myemph{propositional equality}.
The distinction between definitional and propositional equality plays a fundamental role for our purposes.
As usual, the judgements in~(\ref{equ:judge}) may  also be made relative to \myemph{contexts},
which consist of lists of variable declarations of the form
\begin{equation}
\label{equ:context}
\Phi = \big( x_0 \myin A_0,  x_1 \in A_1(x_0), \ldots, x_n \myin A_n(x_0, \ldots, x_{n-1} )\big) \, .
\end{equation}
For the context in~(\ref{equ:context}) to be well-formed, it is necessary that the variables $x_0, \ldots, x_n$ are distinct, and that the following sequence of judgements is derivable
\begin{gather*}
 A_0 \in \type \, ,  \\
( x_0 \in A_0)  \ A_1(x_0) \in \type \, , \\
\dots   \\
( x_0 \in A_0 \, , \ldots \, , x_{n-1} \in A_{n-1}(x_0, \ldots, x_{n-2}))   \ A_n(x_0, \ldots, x_{n-1}) \in \type \, .
\end{gather*}
Whenever we mention contexts, we implicitly assume that they are well-formed. We write $(\Phi) \  \mathcal{J}$ to express that a judgement $\mathcal{J}$ holds under the assumptions
 in~$\Phi$. The axioms for a dependent type theory will be stated here as deduction rules of the form
\[
\begin{prooftree}
(\Phi_1) \ \mathcal{J}_1 \quad \cdots \quad ( \Phi_n) \ \mathcal{J}_n
\justifies
(\Phi) \ \mathcal{J}
\end{prooftree}
\]
From now on, we work with a fixed dependent type theory $\theoryT$. The only assumption that
we make on $\theoryT$ is that it contains the basic axioms stated in Appendix~\ref{app:rules}. 
For more information on dependent type theories,  see~\cite{MartinLofP:inttt,NordstromB:promlt}.
\medskip

  We briefly recall the definition of the classifying category $\classT$ associated to $\theoryT$. For a context $\Phi$ as in~(\ref{equ:context}) and a context $\Psi = (y_0 \in B_0, \ldots, y_m \in B_m(y_0, \ldots, y_{m-1}) )$,
a \myemph{context morphism}  $f : \Phi \rightarrow \Psi$
consists of a  sequence $f = (b_0,  \ldots, b_m)$ such that the following judgements are derivable
\begin{gather*}
(\Phi)   \  b_0 \in B_0 \, , \\
(\Phi)   \    b_1 \in B_1(b_0)  \, ,  \\
    \dots   \\
( \Phi)   \  b_m \in B_m(b_0, \ldots, b_{m-1}) \, .
\end{gather*}
We can define an equivalence relation on contexts by considering two contexts to be equivalent if they coincide
up to renaming of their free variables and up to componentwise definitional equality. Similarly,
we can define an equivalence
relation on context morphisms by considering two context morphisms from $\Phi$ to~$\Psi$
to be equivalent if they concide up to renaming of the free variables in $\Phi$ and up to pointwise definitional equality.
The classifying category~$\classT$ can then be defined as having equivalence classes of contexts as objects,
and equivalence classes of context morphisms as maps. Composition is defined via substitution, 
while identity maps are defined in the evident way~\cite{PittsA:catl}. When working 
with~$\classT$, we implicitly identify contexts and context morphisms
 up to the equivalence relations defined above, without introducing additional notation. The empty context, written $(\, )$ here, is a terminal object in~$\classT$. In what follows, we will be
 particularly interested in the class of context morphisms known as as~\myemph{display maps}~\cite{JacobsB:catltt,TaylorP:prafm}. These are  context morphisms of the form $(\Gamma, x \in A) \rightarrow \Gamma$, defined by projecting away the variable $x \in A$,  where~$\Gamma$ is a context and $A$ is a type relative to $\Gamma$. As we will see in Section~\ref{sec:idtowfs}, display maps play a crucial role in the definition of the identity type weak factorisation system.
\medskip

To work with the category $\classT$ and its slices, it will be convenient to use some of
the abbreviations for manipulating contexts developed in~\cite{deBruijnN:telmtl}. Let us consider a fixed context~$\Gamma$. For a sequence of variable declarations as in~(\ref{equ:context}),  we write $(\Gamma) \ \Phi \in \contype$
 to abbreviate the following sequence of judgements
\begin{gather*}
(\Gamma) \    A_0 \in \type \, ,  \\
(\Gamma, x_0 \in A_0) \    A_1(x_0) \in \type \, , \\
\dots   \\
(\Gamma, x_0 \in A_0,  \ldots,  x_{n-1} \in A_{n-1}( x_0, \ldots, x_{n-1}) ) \  A_n(x_0, \ldots, x_{n-1}) \in \type \, .
\end{gather*}
When these are derivable, we say that $\Phi$ is  a \myemph{dependent context} relative to~$\Gamma$.
Dependent contexts relative to the empty context are simply contexts. If we have a dependent context $\Phi$
relative to $\Gamma$, we obtain  a new context   $(\Gamma, \Phi)$ by concatenation, and an evident map $(\Gamma, \Phi) \rightarrow \Gamma$, projecting away the variables in $\Phi$.
Maps of this form will be referred to as~\myemph{dependent projections}. Note that the class of dependent projections can be seen as  the closure under composition
of the class of display maps.
 For a context~$\Phi$ relative to~$ \Gamma$, as above, and a sequence $a = (a_0, a_1, \ldots, a_n)$, we write~$(\Gamma) \ a \in \Phi$ to abbreviate the following sequence
of judgements
\begin{gather*}
 ( \Gamma) \         a_0 \in A_0 \, , \\
( \Gamma) \  a_1 \in A_1(a_0)  \, , \\
    \dots   \\
 ( \Gamma) \       a_n \in A_n(a_0, \ldots, a_{n-1}) \, .
\end{gather*}
When these can be derived, we say that $a$ is a \myemph{dependent element} of $\Phi$ relative
to~$\Gamma$. Dependent elements of this form are the same thing as context morphisms~$a : \Gamma \rightarrow (\Gamma, \Phi)$ over $\Gamma$.
When $\Gamma$ is the empty context,  we speak of \myemph{global elements} of the context $\Phi$,
which are the same thing as context morphisms $a : (\, ) \rightarrow \Phi$.
The expressions $(\Gamma) \ \Phi \in \contype$ and  $(\Gamma) \  a \in \Phi$ should be understood as counterparts of the first and second judgement in~(\ref{equ:judge}). It is also possible to introduce expressions
$(\Gamma) \ \Phi = \Psi \in \contype$ and $(\Gamma) \ a  = b \in \Phi$ that correspond to the third and fourth  judgement in~(\ref{equ:judge}), respectively, so that the evident counterparts of the rules
in Appendix~\ref{app:rules} hold. The details are essentially straightforward, and hence omitted.

\begin{remark}  \label{thm:notation}
Let $\Phi$ be a context, or a dependent context,   as in~(\ref{equ:context}). We also write~$(x \in \Phi)$ to denote $\Phi$ itself. We  then write $(x \in \Phi  \, ,  y \in \Phi)$ for the result of duplicating $\Phi$ and renaming the variables in
the second copy  to avoid clashes, so as to obtain the context
\begin{multline*}
\big( x_0 \myin A_0,  x_1 \in A_1(x_0), \ldots, x_n \myin A_n(x_0, \ldots, x_{n-1}), \\
y_0 \myin A_0,  y_1 \in A_1(y_0), \ldots, y_n \myin A_n(y_0, \ldots, y_{n-1}) \big) \, .
\end{multline*}
\end{remark}

The  axioms for identity types  are stated in Table~\ref{tab:gene}.
The axiom in~(\ref{equ:form}) is referred to as the \myemph{formation rule} for identity types.
It asserts that if $A$ is a type and $a, b$ are elements of $A$, then $\Id_A(a, b)$ is a type.
We omit the subscript in expressions of the form $\Id_A(a,b)$  when no
confusion arises.
If there exists an element~$p \in \Id(a,b)$, then we say that $a$ and~$b$ are~\myemph{propositionally equal}.
The axiom  in~(\ref{equ:intro}) is referred to as the \myemph{introduction rule} for identity types. Elements of the form
$\refl(a)  \in \Id(a, a)$  are referred
to as \myemph{reflexivity elements}. The axioms  in~(\ref{equ:eli}) and (\ref{equ:comp}) are referred to as the
\myemph{elimination rule} and the  \myemph{computation rule}, respectively. For brevity, we have omitted
from their premisses the judgement
\[
\big(x \myin A, y \myin A, u \myin \Id(x,y),  \Theta(x, y, u)  \big)  \
C(x, y, u)   \myin \type \, .
\]
In these rules we are assuming  that $\Theta(x, y, u)$ is a dependent context relative to the context
$( x \myin A, y \myin A, u \myin \Id(x,y))$. We have highlighted the
variables~$x, y, u$ in~$\Theta(x,y, u)$ in order to describe their role in the deduction rules without
using substitution. All the axioms in Table~\ref{tab:gene} should be understood as being relative to a
context~$\Gamma$ that is common to both the premisses and the conclusion of the rules, which we
leave implicit for brevity. Hence, all the constructions that we perform in $\classT$
may equally well take place in one of its slices.  The presence of the implicit context $\Gamma$
plays a role only when stating three axioms, which we also assume as being part of the rules for identity types, expressing commutation laws between the syntax of identity types and the substitution operation. The first of these axioms is

\[
\begin{small}
\begin{prooftree}
(\Gamma, x \in A) \ B(x) \in \type \   \ 
(\Gamma, x \in A) \ b_1(x) \in B(x) \  \ 
(\Gamma, x \in A) \  b_2(x)  \in B(x)    \ \ 
(\Gamma) \ a \in A
\justifies
(\Gamma) \ \Id_{B(x)}(b_1(x),b_2(x))[a/x] = \Id_{B(a)}(b_1(a), b_2(a)) \in \type
 \end{prooftree}
\end{small}
\medskip
\]
The second of these axioms is

\[
\begin{prooftree}
(\Gamma, x \in A) \ B(x) \in \type \quad
(\Gamma, x \in A) \ b(x) \in B(x) \quad
(\Gamma) \ a \in A
\justifies
(\Gamma) \ \refl(b(x)) [a/x] = \refl( b(a) ) \in \Id_{B(a)}( b(a), b(a)) 
 \end{prooftree}
\medskip
\]
The third axiom, which we do not spell out, asserts a similar commuation property for terms of the form $\myJ(a,b,p, d)$. The assumption of these axioms ensures that all the constructions that we perform, when regarded as taking place in  one of the slices of $\classT$, are stable under pullbacks.

\medskip

\begin{table}[htb]
\begin{boxedminipage}{14cm}
\begin{small}
\medskip

\begin{equation}
\label{equ:form}
\begin{prooftree}
  A \myin \type  \qquad
 a \myin A \qquad
   b \myin A
\justifies
  \Id_A(a,b) \myin \type
\end{prooftree}
\end{equation}
\medskip

\begin{equation}
\label{equ:intro}
\begin{prooftree}
 a \myin A
\justifies
   \refl(a)  \myin \Id_A(a,a)
\end{prooftree}
\end{equation}
\medskip

\begin{equation}
\label{equ:eli}
\begin{prooftree}
 p \myin \Id_A(a,b) \qquad
( x \myin A,   \Theta(x, x, \refl(x))  \big) \   d(x)  \myin
C(x,x,\refl(x))
\justifies
( \Theta(a, b,p)  ) \ \myJ(a,b,p, d) \myin C(a,b,p)
\end{prooftree}
\end{equation}
\medskip

\begin{equation}
\label{equ:comp}
\begin{prooftree}
 a \myin A  \qquad
( x \myin A,  \Theta(x, x, \refl(x)) )  \ d(x)  \myin C(x,x,\refl(x))
\justifies
(  \Theta(a, a, \refl(a))    \  \myJ(a, a, \refl(a) , d ) =
d(a) \myin C(a,a,\refl(a))
\end{prooftree}
\end{equation}
\medskip

\end{small}
\end{boxedminipage}
\bigskip
\caption{Deduction rules for identity types.} \label{tab:gene} \label{tab:subst}
\end{table}

 \begin{remark} \label{rem:meta} Expressions of the form $e(e')$ denote the application  of $e$  to~$e'$ in the metatheory.  For an  expression $e$ in which the variable $x$ may appear free, we write $[x] e$ to denote the $\lambda$-abstraction  of $x$ from $e$ in the metatheory. As usual, we assume the $\eta$-rule and the $\beta$-rule in the metatheory.  Using these rules, it is possible to show that the elimination and computation  rules in~(\ref{equ:eli}) and~(\ref{equ:comp}) can  be stated equivalently as
 \[
 \begin{prooftree}
 p \myin \Id_A(a,b) \qquad
( x \myin A,   \Theta(x, x, \refl(x))  \big) \   e  \myin
C(x,x,\refl(x))
\justifies
( \Theta(a, b,p)  ) \ \myJ(a,b,p, [x] e) \myin C(a,b,p)
\end{prooftree}
\]
and
\[
\begin{prooftree}
 a \myin A  \qquad
( x \myin A,  \Theta(x, x, \refl(x)) )  \ e  \myin C(x,x,\refl(x))
\justifies
(  \Theta(a, a, \refl(a))    \  \myJ(a, a, \refl(a) , [x] e ) =
e[a/x]  \myin C(a,a,\refl(a))
\end{prooftree}
\medskip
\]
where $e[a/x]$ denotes the result of substituting $a$ for $x$ in $e$. We write $[\_] e$ to denote the
$\lambda$-abstraction of a variable that does not appear in $e$.
 \end{remark}

\begin{remark}
The elimination and computation rules, as stated in Table~\ref{tab:gene},  generalise the standard elimination and computation rules for identity types~\cite{NordstromB:marltt}. The latter can be obtained from the former by restricting the context $\Theta(x, y, u)$ to be  empty.  The reason for adopting the generalised rules instead of the standard ones is related to our preference for working without assuming the axioms for $\Pi$-types.  Without
$\Pi$-types, the standard rules are quite weak, since they do not seem to imply the Leibniz rule for
propositional equality~\cite{NordstromB:marltt}, whereas our generalised rules suffice, as shown in Lemma~\ref{thm:leibniz}. Furthermore, the generalised rules become derivable from the standard ones in the presence of $\Pi$-types, so that  our development applies also to dependent type theories with standard axioms for identity types and $\Pi$-types.
 \end{remark}

 \begin{remark}
We do not assume the rules in~(\ref{equ:refl}), to which we refer  as the \myemph{reflection rules}.
\begin{equation}
\label{equ:refl}
\begin{array}{cc}
\begin{prooftree}
 p \myin \Id(a,b)
\justifies
 a = b \myin A
\end{prooftree} \qquad &
\begin{prooftree}
 p \myin \Id(a, b)
\justifies
  \refl(a)  = p \myin \Id(a,b)
\end{prooftree}
\end{array}
\medskip
\end{equation}
The first reflection rule was shown to be independent from the axioms for identity types in~\cite{HofmannM:gromtt}.  Note that the judgement $\refl(a)  \myin \Id(a,b)$, that is presupposed by the conclusion of the second reflection rule, is derivable by the first reflection rule and the standard rules
concerning substitution, as given in Appendix~\ref{app:rules}.
The reflection rules are generally avoided, since they imply that propositional equality  and definitional equality collapse into equivalent notions, which has the effect of destroying the decidability of type-checking, one of the fundamental properties of dependent type  theories~\cite{HofmannM:extcit}.
As shown in~\cite[Proposition 10.1.3]{JacobsB:catltt} extending~\cite[Theorem~1.1]{StreicherT:inviit},
the reflection rules are equivalent to the following rule

\begin{equation*}
\label{equ:eta}
\begin{prooftree}
 p \myin \Id(a,b) \ \
( x \in A, y \myin A, u \myin \Id(x,y),  \Theta(x, y, u) )  \ e(x,y,u) \myin C(x, y, u)
\justifies
 ( \Theta(a,b,p)  )  \  J(a,b,p, [x] e(x,x,\refl(x))) = e(a,b,p)  \in C(a,b,p)
\end{prooftree}
\medskip
\end{equation*}
where we used the notation for  $\lambda$-abstractions in the metatheory explained in Remark~\ref{rem:meta}. If the  computation rule for identity types in~(\ref{equ:eli})  is understood as a version of the~$\beta$-rule, then this  rule can be understood as a version of the $\eta$-rule.
\end{remark}

One fundamental fact for our development is that, as shown
in~\cite{GarnerR:twodmt}, the axioms for identity types can be used to
construct what we will refer to as \myemph{identity contexts}. More precisely,
there are explicit definitions (which we sketch below) such that all the rules
in Table~\ref{tab:ctxid} are derivable. Because of their similarity with the
axioms for identity types, we refer to~(\ref{equ:form2}) as the
\myemph{formation rule}, to~(\ref{equ:intro2}) as the \myemph{introduction
rule}, to~(\ref{equ:eli2}) as the \myemph{elimination rule}, and
to~(\ref{equ:comp2}) as the \myemph{computation rule} for identity contexts.
When stating these rules, we leave again implicit a context $\Gamma$, to which
all the notions are assumed to be relative. For example, in the introduction
rule $\Phi$ may be assumed to be a context relative to $\Gamma$. When stating
the elimination and computation rules, we are assuming that $\Phi$ has the form
in~(\ref{equ:context}) and using the notational conventions set in
Remark~\ref{thm:notation}. For a context~$\Phi$ and $a , b \in \Phi$, we refer
to a context of the form $\Id_\Phi(a,b)$ as an \myemph{identity context}.   As
before,  we have omitted the judgement
\[
\big( x \in \Phi, y \in \Phi,  u \in \Id_\Phi(x,y),  \Theta(x, y, u) \big) \  \Omega(x, y, u)  \myin \contype
\]
from the premisses of the elimination and computation rules. From now on, we
omit the subscript from expressions of the form $\Id_\Phi(a,b)$ if no confusion
arises. It will be convenient to  fix some terminology. When we use  the
elimination rule as in Table~\ref{tab:ctxid}, we will say that we are applying
the elimination rule on~$p \myin \Id(a, b)$. We refer to the relative
context $\Omega(x, y, u)$  as the \myemph{eliminating context}, and to~$d$  as
the \myemph{eliminating family}.

\medskip

\begin{table}[hbt]
\begin{boxedminipage}{13.4cm}

\medskip

\begin{equation}
\label{equ:form2}
\begin{prooftree}
\Phi \in \contype \quad
a \myin \Phi \quad
 b  \myin \Phi
\justifies
  \Id_\Phi(a,b) \myin \contype
\end{prooftree}
\end{equation}
\medskip

\begin{equation}
\label{equ:intro2}
\begin{prooftree}
a \myin \Phi
\justifies
  \refl(a)  \myin \Id_\Phi(a,a)
\end{prooftree}
\end{equation}
\medskip

\begin{equation}
\label{equ:eli2}
\begin{prooftree}
 p \myin \Id_\Phi(a,b) \quad
( x \in \Phi,
\Theta(x,x, \refl(x))  \   d(x) \in  \Omega(x,x, \refl(x))
 \justifies
(  \Theta(a, b, p)  ) \  \myJ(a,b, p, d) \myin \Omega(a, b, p)
\end{prooftree}
\end{equation}
\medskip

\begin{equation}
\label{equ:comp2}
\begin{prooftree}
 a \myin \Phi \quad
( x \in \Phi,  \Theta(x,x, \refl(x)))   \ d(x) \myin \Omega(x,x, \refl(x))
\justifies
(  \Theta(a, a,  \refl(a) ))  \
\myJ(a,a, \refl(a), d) = d(a) \myin \Omega(a,a, \refl(a))
\end{prooftree}
\end{equation}
\medskip

\end{boxedminipage}
\medskip
\caption{Deduction rules for identity contexts.} \label{tab:ctxid}
\end{table}

We construct the identity context of a context $\Phi$ by induction on its
length. For a context of length $n =1$, the definition is straightforward: if
$\Phi = (x \in A)$, then elements of $\Phi$ are the same as elements of the
type $A$, and the identity context corresponding to $a, b \in \Phi$ is given by
\[
\Id_\Phi(a,b) = ( u \in \Id_A(a,b) ) \, .
\]
The introduction rule~(\ref{equ:intro2}) then reduces to the identity type
introduction rule~(\ref{equ:comp}); whilst the elimination
rule~(\ref{equ:eli2}) may be derived from the identity type elimination
rule~(\ref{equ:eli}) by induction on the length of the context $\Omega$.
Assuming we know how to define the identity contexts associated to contexts of
length~$n$, we can  describe the identity context associated to a context of
length~$n+1$ as follows. First we use the inductive hypothesis to prove
Lemma~\ref{thm:leibniz} below. This Lemma states a very useful property, which
we refer as the \myemph{Leibniz rule for contexts}, that is going to be used
repeatedly in what follows. To prove it, we make use of the elimination and
computation rules~(\ref{equ:eli2}) and~(\ref{equ:comp2}) for contexts of 
length~$n$.

\begin{lemma}  \label{thm:leibniz} \label{thm:leibnizcomp}  For each context $\Phi$ of length
$n$, we can derive a rule of the form

\[
\begin{prooftree}
  p \in \Id(a,b) \qquad
  (x \in \Phi) \ \Omega(x) \in \contype \qquad
 e \myin \Omega(a)
\justifies
p_!(e) \myin \Omega(b)
\end{prooftree}
\]
such that

\[
\begin{prooftree}
a \myin \Phi \quad
e \myin \Omega(a)
\justifies
(\refl(a))_{!}(e) = e \myin \Omega(a)
\end{prooftree}
\medskip
\]
holds.
\end{lemma}

\begin{proof}
We use elimination over $p \myin \Id(a,b)$ with
\[
\big( x \in \Phi, y \in \Phi, u \in \Id(x,y), z \in \Omega(x) \big) \
\Omega(y)  \in \contype
\]
as the eliminating context. Since we have
\[
(x \in \Phi, z \in \Omega(x)) \ z \in \Omega(x)
\]
the elimination rule allows us to derive
\[
(z \in\Omega(a)) \ \myJ(a,b,p,[x]  z) \in \Omega(b) \, .
\]
The required term $p_{!}(e)$ is defined as the result of substituting  $e \myin
\Omega(a)$ for~$z \myin \Omega(a)$ in the expression~$\myJ(a,b,p, [x]  z)$, so
that
\[
p_{!}(e) =  \myJ(a,b,p, [x]  e) \in \Omega(b)
\]
The second rule  is an immediate consequence of this definition and the
computation rule.
\end{proof}

We now define the identity context of a context of length $n+1$. Let us assume
to have a context $\Phi$ of length $n$ as in~(\ref{equ:context}), and consider
a context $\Phi'$ of length $n+1$ of the form
\begin{equation}
\label{equ:nplus}
\Phi' =  (x \in \Phi,  x_{n+1} \in A_{n+1}(x)) \, .
\end{equation}
By definition, elements $a', b' \in \Phi'$ have the form $a' = (a, a_{n+1})$
and $ b' = (b, b_{n+1})$, where $a, b \in \Phi$, $a_{n+1} \in A_{n+1}(a)$, and
$b_{n+1} \in A_{n+1}(b)$. We define their associated identity context to have
the form
\[
\Id(a',b') = \big( u \in \Id_{\Phi}(a,b), u_{n+1} \in \Id_{A_{n+1}(b)}(u_{!}(a_{n+1}), b_{n+1})  \big)
\]
where we have $u_{!}(a_{n+1}) \in A_{n+1}(b)$ by the following application of
the Leibniz rule
\[
\begin{prooftree}
  u \in \Id_\Phi(a,b) \qquad
  (x \in \Phi) \ A_{n+1}(x) \in \type \qquad
 a_{n+1} \myin A_{n+1}(a)
\justifies
u_{!}(a_{n+1}) \myin A_{n+1}(b)
\end{prooftree}\ \ \text.
\medskip
\]
This gives us the formation rule~(\ref{equ:form2}) for $\Id_{\Phi'}$. The corresponding introduction, elimination and computation rules are defined in
a similar inductive manner. Since in the following we will need only the
description of the identity contexts and the rules in Table~\ref{tab:ctxid}, we
omit the remaining details, which may be found in~\cite{GarnerR:twodmt}.

\section{The fundamental groupoid of a context}
\label{sec:fungpd}

The axioms for identity types allow us  to think of a context $\Phi$ as  a space, as we explain below.
To emphasize this perspective, we refer to elements of $\Phi$ as \myemph{points}. Given two
points~$a, b \in \Phi$, we refer to elements of
$\Id(a,b)$ as \myemph{paths} from $a$ to $b$. Following this
idea, the context~$\Id(a, b)$ can  be thought of as the space of all
paths from $a$ to $b$.  To
speak of homotopies between paths, it suffices to apply the
formation rule as follows:

\[
\begin{prooftree}
 p_0 \in \Id(a, b) \quad
 p_1 \in \Id(a,b)
\justifies
 \Id(p_0,p_1) \in \contype
\end{prooftree}
\medskip
\]
Elements $\theta \in \Id(p_0,p_1)$ will be referred to as \myemph{homotopies} from $p_0$ to
$p_1$, and two paths $p_0 \in \Id(a,b)$ and $p_1 \in \Id(a,b)$ will be said to
be~\myemph{homotopic} if there exists a homotopy between them.
Our aim is to pursue this analogy and to define a groupoid having the
points of~$\Phi$ as objects and  homotopy equivalence classes of paths as maps.
Because of the analogy with the definition of the fundamental groupoid of a
space~\cite{MayP:concat}, we refer to this groupoid as the
\myemph{fundamental groupoid} of a context. \medskip

Lemma~\ref{thm:sym} below can the
understood as expressing that paths can be composed, that there is a trival
path on each point, and that paths can be reversed.
From now on, we assume to work with a fixed context~$\Phi$ as in~(\ref{equ:context}).
We use freely the notation introduced in Remark~\ref{thm:notation}.

\begin{lemma} \label{thm:sym} \label{thm:trans}  We can derive rules of the form

\[
\begin{array}{ccc}
\begin{prooftree}
p \in \Id(a,b) \quad
q \in \Id(b,c)
\justifies
q \circ p \in \Id(a,c)
\end{prooftree} \qquad &
\begin{prooftree}
a \in \Phi
\justifies
1_{a} \in \Id(a,a)
\end{prooftree} \qquad &
\begin{prooftree}
p \in \Id(a,b)
\justifies
p^{-1} \in \Id(b,a)
\end{prooftree}
\end{array}
\medskip
\]
such that
\[
\begin{array}{cc}
\begin{prooftree}
p \in \Id(a,b)
\justifies
1_{b} \circ p  = p \in \Id(a, b)
\end{prooftree} \qquad &
\begin{prooftree}
a \in \Phi
\justifies
(1_{a})^{-1} =  1_{a}  \in \Id(a, a)
\end{prooftree}
\end{array}
\]
hold.
\end{lemma}

\begin{proof} For the first rule, we apply the Leibniz rule as follows

\[
\begin{prooftree}
q \in \Id(b,c) \quad
(x \in \Phi) \  \Id(a,x) \in \contype \quad
p \in \Id(a,b)
\justifies
q_{!}(p) \in \Id(a,c)
\end{prooftree}
\medskip
\]
and define $q \circ p \defeq q_{!}(p) \in \Id(a,c)$. For the second rule, we use the introduction rule,
and simply define $1_{a} \defeq \refl(a)  \in \Id(a,a)$.
For the third rule, we apply the elimination rule on $p \in \Id(a,b)$, with
\[
(x \in \Phi, y \in \Phi, u \in \Id(x, y) ) \ \Id(y,x) \in \contype
\]
as the eliminating context, and $\refl(x) \in \Id(x,x)$ as the eliminating family. Hence, we can define
\[
p^{-1} \defeq \myJ(a,b,p, [x] \refl(x))  \in \Id(b,a)
\]
The computation rule implies the other rules.
\end{proof}

\begin{lemma} \label{thm:compat}  Let $p_0, p_1 \in \Id(a,b)$ and $q_0, q_1 \in \Id(b,c)$.
We can derive rules of the form

\[
\begin{array}{cc}
\begin{prooftree}
\phi \in \Id(p_0,p_1) \quad
\psi \in \Id(q_0,q_1)
\justifies
\psi \circ \phi \in \Id(q_0 \circ p_0, q_1 \circ p_1)
\end{prooftree} \qquad &
\begin{prooftree}
\phi \in \Id(p_0,p_1)
\justifies
\phi^{*} \in \Id({p_0}^{-1},{p_1}^{-1})
\end{prooftree}
\end{array}
\medskip
\]
such that

\[
\begin{array}{cc}
\begin{prooftree}
\phi \in \Id(p_0, p_1) \quad
q \in \Id(b,c)
\justifies
1_{q} \circ \phi = \phi \in \Id(q \circ p_0, q \circ p_1)
\end{prooftree} \qquad &
\begin{prooftree}
p \in \Id(a,b)
\justifies
(1_p)^* = 1_p \in \Id(p,p)
\end{prooftree}
\end{array}
\medskip
\]
hold.
\end{lemma}

\begin{proof} For the first rule, use elimination over
 $\phi \in \Id(p_0, p_1)$ and $\psi
\in \Id(q_0,q_1)$, and Lemma~\ref{thm:sym}. For the second rule,
use elimination over $\phi \in \Id(p_0,p_1)$ and Lemma~\ref{thm:sym}.
\end{proof}

\begin{lemma} \label{thm:higher} \label{thm:iota}  We can derive rules of the form

\medskip
\[
\begin{prooftree}
p \in \Id(a,b) \qquad
q \in \Id(b,c) \qquad
r \in \Id(c,d)
\justifies
\alpha_{p,q,r} \in \Id( (r \circ q) \circ p, r \circ (q \circ p))
\end{prooftree}
\]
\medskip
\[
\begin{array}{cc}
\begin{prooftree}
p \in \Id(a,b)
\justifies
\phi_p \in \Id(1_{b} \circ p, p)
\end{prooftree}   \qquad &
\begin{prooftree}
p \in \Id(a,b)
\justifies
\psi_p \in \Id(p \circ 1_{a},  p)
\end{prooftree}
\end{array}
\]
\medskip
\[
\begin{array}{cc}
\begin{prooftree}
p \in \Id(a,b)
\justifies
\sigma_p \in \Id(p^{-1} \circ p, 1_{a})
\end{prooftree}  \qquad &
\begin{prooftree}
p \in \Id(a,b)
\justifies
\tau_p \in \Id(p \circ p^{-1},  1_{b})
\end{prooftree}
\end{array}
\medskip
\]
such that

\[
\begin{prooftree}
p \in \Id(a,b) \qquad
q \in \Id(b,c)
\justifies
\alpha_{p, q, 1_{c}}  = 1_{q \circ p} \in \Id(q \circ p, q \circ p)
\end{prooftree}
\]
\medskip

\[
\begin{array}{cc}
\begin{prooftree}
a \in \Phi
\justifies
\phi_{1_a} =  1_{1_a} \in \Id(1_a, 1_a)
\end{prooftree}  \qquad &
\begin{prooftree}
a \in \Phi
\justifies
\psi_{1_a}  = 1_{1_a}  \in \Id(1_a, 1_a)
\end{prooftree}
\end{array}
\medskip
\]

\[
\begin{array}{cc}
\begin{prooftree}
a \in A
\justifies
\sigma_{1_a} =  1_{1_a} \in \Id(1_a, 1_a)
\end{prooftree}  \qquad &
\begin{prooftree}
a \in A
\justifies
\tau_{1_a}  = 1_{1_a}  \in \Id(1_a, 1_a)
\end{prooftree}
\end{array}
\]
hold.
\end{lemma}

\begin{proof} For $\alpha_{p,q,r}$ use elimination over $r \in \Id(c,d)$. For
$\phi_p$ we can define $\phi_p$ to
be~$1_p$ by Lemma~\ref{thm:sym}. For $\psi_p$ use elimination over
$p \in \Id(a,b)$. For $\sigma_p$ and $\tau_p$ use elimination over $p \in \Id(a,b)$.
\end{proof}

Let $a, b \in \Phi$. An application of Lemma~\ref{thm:sym}, taking  $\Phi$ therein to be $\Id_\Phi(a,b)$,
shows that homotopy of paths is a reflexive, symmetric, and transitive relation on the set of paths
from $a$ to $b$. We write~$[a,b]$ for the quotient set,
and $[p] : a \rightarrow b$ for the equivalence class of a path $p \in \Id(a,b)$.
We can now define the fundamental groupoid~$\mathcal{F}(\Phi)$  associated to  $\Phi$. The objects of~$\mathcal{F}(\Phi)$ are the global elements of $\Phi$. The maps from $a$ to $b$ in $\mathcal{F}(\Phi)$ are equivalence classes of paths
$[p] : a \rightarrow b$. Composition, identities, and inverses in $\mathcal{F}(\Phi)$ are defined by letting
 \[
[ q ] \circ [p] \defeq [ q \circ p] \, , \quad
1_{a} \defeq [1_a] \, , \quad
 [p]^{-1} \defeq [p^{-1}] \, .
\]
These operations are well-defined by Lemma~\ref{thm:compat}. To establish the  axioms for a category, we need
to show
\begin{equation}
\label{equ:category}
[r \circ (q \circ p) ] = [ (r \circ q) \circ p] \, ,  \quad
[ 1_{b} \circ p ] = [p] \, , \quad
[p \circ 1_{a}] = [p]
\end{equation}
and to establish the additional axioms for a groupoid, we need to show
\begin{equation}
\label{equ:groupoid}
[ p^{-1} \circ p] = [1_{a}]  \, , \quad
[ p \circ p^{-1}] = [1_{b}] \, .
\end{equation}
For~(\ref{equ:category}), we need homotopies $\alpha \in \Id(r \circ (q \circ p), (r \circ q) \circ p)$,
$\phi \in \Id(1_{b} \circ p, p)$, and $\psi \in \Id(p \circ 1_{a}, p)$.
For~(\ref{equ:groupoid}), we need homotopies $\sigma \in \Id(p^{-1} \circ p, 1_{a})$
and~$\tau \in \Id(p \circ p^{-1}, 1_{b})$.  All of these are provided by Lemma~\ref{thm:higher}.
We now show that this construction extends to a functor. Recall 
that we write $\Gpd$ for the category of small groupoids and functors.

\begin{proposition} \label{thm:functor} Let $\theoryT$ be a dependent type theory with axioms for identity types.
The function mapping a context $\Phi$ to its fundamental groupoid $\mathcal{F}(\Phi)$
extends to a functor $\mathcal{F} : \classT \rightarrow \Gpd$.
\end{proposition}

\begin{proof} We need to define a functor $\mathcal{F}(f) : \mathcal{F}(\Phi) \rightarrow \mathcal{F}(\Psi)$ for every
 context morphism $f : \Phi \rightarrow \Psi$. On objects, $\mathcal{F}(f)$ sends $a \in \mathcal{F}(\Phi)$ to $f(a) \in
\mathcal{F}(\Psi)$. On maps, $\mathcal{F}(f)$ sends $[p] : a \rightarrow b$ to $[f(p)] : f(a) \rightarrow f(b)$, where
$f(p) \in \Id(f(a), f(b))$ is defined using the elimination rule by letting
\begin{equation}
\label{equ:fp}
f(p) = \myJ(a,b,p, [x] 1_{fx}) \in \Id(f(a), f(b)) \, ,
\end{equation}
so that
\begin{equation}
\label{equ:fid}
f(1_a) = 1_{fa} \in \Id(f(a), f(a)) \, .
\end{equation}
It is routine to check that the action of $\mathcal{F}(f)$ on maps is well-defined. To show that $\mathcal{F}(f)$ is a
functor amounts to verifying the equations
\[
   [f(q \circ p)] = [f(q) \circ f(p)] \, , \qquad   [f(1_a)] = [1_{f(a)}]   \, .
\]
For the first equation, elimination on $q \in \Id(b,c)$ can be used to
exhibit the required homotopy between $f(q \circ p)$ and $f(q) \circ f(p)$.  For the second equation, use~(\ref{equ:fid}).
We have therefore defined $\mathcal{F} : \classT \rightarrow \Gpd$ on objects and maps. Thus, it
remains to check that it is a functor. We begin by checking
\begin{equation*}
\mathcal{F}(g \circ f) = \mathcal{F}(g) \circ  \mathcal{F}(f) \, .
\end{equation*}
It is clear  that $\mathcal{F}(g \circ f)$ and $\mathcal{F}(g) \circ  \mathcal{F}(f)$ have the same action on objects.
To show that they coincide on maps, we need to show that we have a homotopy between
$g(f(p))$ and $(gf)(p)$ for every $p \in \Id(a,b)$. By (\ref{equ:fp}), we have
\[
g(f(p)) = \myJ(f(a), f(b), f(p), [x] \, 1_{g(x)} ) \in \Id(gf(a), gf(b))
\]
and
\[
(gf)(p) = \myJ(a, b, p, [x]  \, 1_{gf(x)}) \in \Id(gf(a), gf(b)) \, .
\]
The required homotopy can be obtained by  the elimination rule on~$p \in \Id(a,b)$. To conclude the proof, it suffices to check that
\[
\mathcal{F}(1_\Phi) = 1_{\mathcal{F}(\Phi)}   \, .
\]
As before, it is clear that $\mathcal{F}(1_\Phi)$ is the identity on objects. To check that it is the identity on maps,
it suffices to show that $[p] : a \rightarrow b$ and $[1_\Phi(p)] : a \rightarrow b$ are the same equivalence classes,
which can be proved by the elimination rule on~$p \in \Id(a,b)$.
\end{proof}

\section{The identity type  weak factorisation system}
\label{sec:idtowfs}

Let us recall the notion of a weak factorisation system~\cite{BousfieldA:confsc}. For this, we need some terminology and notation.
Let $\catE$ be a category. Given maps  $f : A \rightarrow B$ and $g : C \rightarrow D$ in $\catE$, we say that~$f$ has the \emph{left lifting property}  with respect to $g$, or that $g$ has the \emph{right lifting property}  with respect to~$f$, if every  commutative diagram of the form
\begin{equation*}
\xymatrix{
A \ar[r]^h \ar[d]_f & C \ar[d]^g \\
B \ar[r]_k   &  D }
\end{equation*}
has a \myemph{diagonal filler}, that is to say is a map $j : B \rightarrow C$ making the diagram
\begin{equation*}
\xymatrix{
A \ar[r]^h \ar[d]_f  & C \ar[d]^g \\
B \ar[r]_k \ar[ur]^{j}  & D }
\end{equation*}
commute. We write $f \pitchfork g$ to denote this situation. For a class of maps~$\mapM$, we define $\Mpitch$ to be the class of maps having the right lifting  property with respect to every map in $\mapM$.
 Similarly, we
 define $\pitchM$ to be the class of maps having the left lifting property with respect to every map in
 $\mapM$. A \emph{weak factorisation system} on  $\catE$ consists of a pair of
classes of maps $(\mapA, \mapB)$ such that the following hold.
\begin{enumerate}
\item \label{itm:fact} Every map $f $ admits a factorisation $f = pi$ with $i \myin \mapA$ and $p \myin \mapB$.
\item \label{itm:wko} $\mapA^\pitchfork = \mapB$ and $\mapA= {}^\pitchfork \mapB$.
 \end{enumerate}
We refer to (\ref{itm:fact}) as the Factorisation Axiom and to (\ref{itm:wko}) as the Weak Orthogonality Axiom.  Elements of $\mapA$ and $\mapB$ will be also referred to as $\mapA$-maps and
$\mapB$-maps, respectively. For more information on weak factorisation systems, see~\cite[Appendix~D]{JoyalA:theqca}.
 \medskip

\begin{table}[htb]
\begin{boxedminipage}{13.2cm}
\begin{footnotesize}
\bigskip
\begin{equation*}
\begin{prooftree}
 \Phi \myin \contype
\justifies
(x \in \Phi, y \in \Phi)  \ \Id_\Phi(x,y) \myin \contype
\end{prooftree}
\end{equation*}
\bigskip
\begin{equation*}
\begin{prooftree}
 \Phi \myin \contype
\justifies
(x \in \Phi) \ \refl(x) \myin \Id_\Phi(x,x)
\end{prooftree}
\end{equation*}
\bigskip
\begin{equation*}
\begin{prooftree}
( x \in \Phi,  \Theta(x, x, \refl(x))) \    d(x) \myin \Omega(x,x, \refl x)
\justifies
( x \in \Phi, y \in \Phi, u \in \Id_\Phi(x,y), \Theta(x,y,u)  ) \  \myJ(x,y,u, d) \myin
\Omega(x,y,u)
\end{prooftree}
\end{equation*}
\bigskip
\begin{equation*}
\begin{prooftree}
(  x \in \Phi,  \Theta(x, x, \refl(x))  ) \ d(x) \myin \Omega(x, x, \refl(x))
\justifies
\big(  x \in \Phi, \Theta(x, x \refl(x)) \big)  \ \myJ(x, x, \refl(x),d) = d(x) \myin \Omega(x, x, \refl(x))
\end{prooftree}
\end{equation*}
\end{footnotesize}
\end{boxedminipage}
\bigskip
\caption{Variable-based rules for identity contexts.} \label{tab:varonly}
\end{table}

Let us now return to consider the dependent type theory $\theoryT$ and  its classifying category $\classT$. Recall from Section~\ref{sec:identitytypes} that a display map is a
context morphism of the form $(\Gamma, x \in A) \rightarrow \Gamma$, obtained by forgetting
the variable $x \in A$. We write  $\mathcal{D}$  for the set of display maps in $\classT$. Our main result is the following.

\begin{theorem} \label{thm:wfs}
Let $\theoryT$ be a dependent type theory with axioms for
identity types. The pair $(\mapA, \mapB)$, where $\mapA \defeq {}^\pitchfork \mathcal{D}$
and~$\mapB \defeq \mapA^\pitchfork$, forms a weak factorisation system on $\classT$.
\end{theorem}

Let us emphasize that identity types are not involved in the definition of the classes of maps $\mapA$ and $\mapB$.
They are, however, used extensively in the proof that these classes of maps satisfy the axioms for a weak factorisation system. As usual in the proof of the existence of a weak factorisation system, the difficulties are concentrated in  one particular step of the proof.  Lemma~\ref{thm:factone} is the key step in our case, with
the proof of Theorem~\ref{thm:wfs} following from it by standard arguments in the theory of
weak factorisation systems.
To prove Lemma~\ref{thm:factone}, it is convenient to work with the equivalent formulation of the rules for identity contexts given in Table~\ref{tab:varonly}. The equivalence between the sets of rules in Table~\ref{tab:ctxid} and in Table~\ref{tab:varonly} follows by standard properties of substitution. As before, we use the
 notational convention stipulated in Remark~\ref{thm:notation}. Also recall that a dependent
 projection is a context morphism of the form $(\Gamma, \Phi) \rightarrow \Gamma$, defined
 by projecting away the variables in $\Phi$, where $\Gamma$ is a context and $\Phi$ is
 context relative to $\Gamma$.

\begin{lemma}  \label{thm:factone}  Every map $f$ admits a factorisation $f = p i$, where
$i \in \mapA$  and $p$ is a dependent projection.
\end{lemma}

\begin{proof} For  $f : \Phi \rightarrow \Psi $, define $\Id(f) \defeq \big( x \in \Phi,  y \in  \Psi,
u \in \Id_\Psi(fx, y) \big)$.
The required factorisation is defined as follows:
\begin{equation}
\label{equ:fact}
\xymatrix{
\Phi  \ar[r]^(.43){i_f}  &  \Id(f)  \ar[r]^{p_f} & \Psi \, ,}
\end{equation}
where $i_f \defeq (x, fx, 1_{fx})$ and $p_f = (y)$. Apart from the ordering of the variable declarations $x \in \Phi$
and $y \in \Psi$, which is clearly unessential,  $p_f$ is a dependent
projection, as required. Hence, we only need to show that $i_f \in \mapA$. This amounts to showing that it has the left lifting property with respect to all display maps.  This amounts to providing diagonal fillers for every
diagram of the form
\[
\xymatrix{
\Phi \ar[r] \ar[d] &  (v \in \Lambda, z \in D(v)) \ar[d]  \\
\Id(f) \ar[r] &  (v \in \Lambda) }
\]
Since display maps are closed under pullback~\cite[Lemma~6.3.2]{PittsA:catl},  it suffices to show that we can define a diagonal filler for every diagram of the form
\[
\xymatrix{
\Phi  \ar[r] \ar[d] & (\Id(f), z \in C(x,y,u)) \ar[d] \\
\Id(f) \ar@{=}[r]  & \Id(f) }
\]
The right-hand side display map gives us
a dependent type $C(x, y, u)$ relative to $\Id(f) = (x \in \Phi, y \in \Psi, u \in \Id(f))$.
By the commutativity of the diagram,
the top horizontal map gives us
a dependent element  $d(x) \in C(x, fx,  1_{fx})$ relative to $(x \in \Phi)$. We can derive
\[
(x \in \Phi, y_0 \in \Psi, y_1 \in \Psi, v \in \Id(y_0, y_1), \Theta(y_0, y_1, v) ) \ C(x, y_1, v \circ u)  \in \type \, ,
\]
where  $\Theta(y_0, y_1, v) = (u \myin \Id(fx, y_0), z \myin C(x, y_0, u) )$.
By the definitional equality $1_y \circ u = u \in \Id(f(x), y_0)$, proved in Lemma~\ref{thm:sym},  we have
\[
\big(x \in \Phi, y \in \Psi,  \Theta(y,y, 1_y) \big) \ z \in C(x,y, 1_y \circ u) \, ,
\]
By the elimination rule applied to $v \in \Id(y_0, y_1)$, we obtain
\[
(x \in \Phi, y_0 \in \Psi, y_1 \in \Psi, v \in \Id(y_0, y_1), \Theta(y_0, y_1, v) ) \   \myJ(y_0, y_1, v, [\_] z)
\in C(x, y_1, v \circ u)
\]
We then obtain
\[
\big(x \in \Phi, y \in \Psi, u \in \Id(fx, y), z \in  C(x, fx, 1_{fx} )  \big) \ n(x,y,u,z) \in C(x, y, u \circ 1_{fx})
\]
where $n(x,y,u,z) \defeq \myJ(fx, y, u, z)$. We can now substitute $d(x)$ for $z$ and obtain
\[
(x \in \Phi, y \in \Psi, u \in \Id(fx, y)) \ n(x,y, u, d(x)) \myin C(x, y,  u \circ 1_{fx})
\]
Let us now recall that by Lemma~\ref{thm:higher} we have $\psi_{u} \myin \Id(u \circ 1_{fx}, u)$. Hence, we can
apply Lemma~\ref{thm:leibniz} and obtain
\[
(x \in \Phi, y \in \Psi, u \in \Id(fx, y)) \ (\psi_u)_{!}(n(x,y, u, d(x))) \in C(x, y, u)
\]
We claim that $j \defeq (x, y,  u, (\psi_u)_{!}(n(x,y, u, d(x))))$ provides the required filler, fitting in the diagram

\[
\xymatrix{
\Phi  \ar[r] \ar[d] & (\Id(f), z \in C(x,y,u)) \ar[d] \\
\Id(f) \ar@{=}[r]  \ar[ur]^{j} & \Id(f) }
\]
The commutativity of the bottom triangle is evident. The commutativity of the top triangle follows by the
chain of definitional equalities
\begin{eqnarray*}
(\psi_{1_{fx}})_{!}(n(x, fx, 1_{fx}, d(x) )) & = &   n(x, fx, 1_{fx}, d(x) ) \\
   & = & \myJ(fx, fx, 1_{fx}, [\_] d(x) ) \\
   & = & d(x)    \, .
   \end{eqnarray*}
 Here, we used Lemma~\ref{thm:sym}, the definition of $n$, and the elimination rule for identity types.
  The commutativity of the bottom triangle is immediate.
\end{proof}

\begin{proof}[Proof of Theorem~\ref{thm:wfs}]  Since $\mathcal{D} \subseteq \mapB$ and $\mathcal{B}$ is closed under composition, $\mathcal{B}$ contains all dependent projections. The Factorisation Axiom then follows from Lemma~\ref{thm:factone}. For the Weak Orthogonality Axiom,
observe that  $\mapA^\pitchfork = \mapB$ by the very definition of $\mapB$ in Theorem~\ref{thm:wfs}.
So, we only need to show that $\mapA = {}^{\pitchfork} \mapB$.
Since~$\mathcal{D} \subseteq \mapB$,  we have ${}^\pitchfork \mapB \subseteq {}^\pitchfork \mathcal{D}$, and so~${}^\pitchfork \mapB \subseteq \mapA$. Thus, it remains to prove that $\mapA \subseteq {}^\pitchfork \mapB$.  Observe that every map in
$\mapA$ has the left lifting property with respect to every dependent projection. This is because maps in
$\mapA$ have the left lifting property with respect to display maps, and dependent projections are composites of display maps. The required inclusion~$\mapA \subseteq {}^\pitchfork \mapB$ amounts
to showing  that every map in
$\mapA$ has the left lifting property with respect to every map in $\mapB$.
 This follows from the fact that every map in
$\mapA$ has the left lifting property with respect to every dependent projection, which we observed
above,
 and the fact that
every map in~$\mapB$ is a retract of a dependent projection, which follows from Lemma~\ref{thm:factone} and the Retract Argument~\cite[Lemma~1.1.9]{HoveyM:modc}.  \end{proof}

Let us  illustrate what happens when we apply the factorisation of Lemma~\ref{thm:factone} to
the identity map $1_\Phi :  \Phi \rightarrow \Phi$. The factorisation in~(\ref{equ:fact}) becomes
\[
\xymatrix{
\Phi  \ar[r]^(.4){i}  & \Id(\Phi)  \ar[r]^(.56){p} & \Phi  \, ,}
\]
where $\Id(\Phi) \defeq (x \in \Phi, y \in \Phi, u \in \Id(x,y))$, and the
maps~$i$ and $p$ are defined by letting $i \defeq (x, x, 1_x)$ and $p \defeq  (y)$, respectively.
Let us consider a diagram of the form
\begin{equation*}
\xymatrix{
\Phi \ar[rr] \ar[d]_{i}  &  & (\Id(\Phi),  z \in C(x,y,u) ) \ar[d] \\
 \Id(\Phi)  \ar@{=}[rr]    & &   \Id(\Phi) }
\end{equation*}
where $C(x,y,u)$ is a type relative to the context $\Id(\Phi)$.
By the commutativity of the diagram, the top horizontal map gives us a dependent element
$d(x) \in C(x, x,  \refl(x))$ relative to~$(x \in \Phi)$. By the elimination rule, we can deduce that  $\myJ(x,y,u,d) \in C(x,y,u)$. We can  therefore define a filler $j$
\begin{equation*}
\xymatrix{
\Phi \ar[rr] \ar[d]_{i}  & &  (\Id(\Phi),  z \in C(x,y,u)) \ar[d] \\
 \Id(\Phi)  \ar@{=}[rr]   \ar[urr]^{j} &  &  \Id(\Phi) }
\end{equation*}
by letting $j \defeq  (x,y,u, \myJ(x,y,u,d))$.  The commutativity of the top triangle follows from the computation rule, while the commutativity of the bottom triangle is immediate. This is the key idea underpinning the semantics of identity types
in weak factorisation systems introduced in~\cite{AwodeyS:homtmi}.

\begin{remark} The identity type weak factorisation system is not a functorial
weak factorisation system. This would entail having operations mapping
commutative squares
\[
\xymatrix{
\Phi \ar[r]^{f}  \ar[d]_{u}  & \Psi \ar[d]^{v} \\
\Phi' \ar[r]_{g}  & \Psi' }
\]
into commutative diagrams
\begin{equation}
\label{equ:funcnot}
\vcenter{\hbox{\xymatrix{
\Phi \ar[rr]^{f}  \ar[d]_{i_u}  & &  \Psi \ar[d]^{i_v}  \\
\Id(u) \ar[rr]^{\Id_{f,g}} \ar[d]_{p_u}  & &  \Id(v) \ar[d]^{p_v}  \\
\Phi' \ar[rr]_{g} & &  \Psi' }}}
\end{equation}
subject to the functoriality conditions 
\begin{equation}
\label{equ:functoriality}
\Id_{1_\Phi, 1_{\Phi'}} = 1_{\Id(u)} \qquad
\Id_{h, k} \circ \Id_{f, g} = \Id_{hf, kg}
\end{equation}
We can use the identity
elimination rules to define maps \mbox{$\Id_{f,g} : \Id(u) \rightarrow \Id(v)$}
making the diagram in~(\ref{equ:funcnot}) commute and satisfying the 
first functoriality condition in~(\ref{equ:functoriality}). However, these maps need not satisfy the second functoriality condition in~(\ref{equ:functoriality}) strictly,
but only up to propositional equality. To see why this is, let us consider the
following situation:
\[
\xymatrix{
\Phi \ar[r]^f \ar@{=}[d] & \Psi \ar[r]^g \ar@{=}[d]  & \Xi \ar@{=}[d] \\
\Phi \ar[r]_f & \Psi \ar[r]_g  & \Xi\text. }
\]
The maps $\Id_{f,f} : \Id(\Phi) \rightarrow \Id(\Psi)$ and $\Id_{g,g} :
\Id(\Psi) \rightarrow \Id(\Xi)$ are defined as in equation~(\ref{equ:fp}) of
Proposition~\ref{thm:functor}. Their composite will not in general be equal
to the map $\Id_{gf, gf} : \Id(\Phi) \rightarrow \Id(\Xi)$, but only pointwise
propositionally equal.
\end{remark}

\section{Characterisation and applications}

We provide an explicit characterisation of the maps in the classes $\mapA$ and $\mapB$ of the
identity type weak factorisation system established in Theorem~\ref{thm:wfs}.
For this, we introduce some terminology, which is inspired by concepts of
2-dimensional category theory~\cite{BlackwellR:twodmt,LackS:homta2,StreetR:fibys}.
We define a context morphism $f : \Phi \rightarrow \Psi$ to
be a \myemph{type-theoretic injective equivalence} if we can derive a jugdement
\[
( y \in \Psi ) \  s(y) \in \Phi
\]
such that we can derive also judgements of the form
\begin{gather}
(x \in \Phi ) \  x =  s(f(x))   \in \Phi \, , \label{equ:injeqv1} \\
(y \in \Psi ) \  \varepsilon_y \in \Id_\Psi(f(s(y)), y) \label{equ:injeqv2}  \, , \\
( x \in \Phi ) \  \varepsilon_{f(x)} = 1_{f(x)} \in \Id_\Psi(f(x), f(x))     \label{equ:injeqv3} \, .
\end{gather}
We say  that $f: \Phi \rightarrow \Psi$ is a \myemph{type-theoretic normal isofibration} if we can derive a judgement
\[
\big( x \in \Phi, y \in \Psi, u \in \Id(f(x),y) \big)  \ j(x,y,u) \in \Phi
\]
such that we can also derive judgements of the following form
\begin{gather}
\big( x \in \Phi, y \in \Psi, u \in \Id(f(x),y) \big) \  f j(x,y,u) = y \in \Psi \, , \label{equ:fib1} \\
( x \in \Phi ) \   j(x, fx, 1_{fx} ) = x \in \Phi \, \label{equ:fib2} .
\end{gather}
Although the identity type weak factorisation system does not seem to be functorial, we can follow the argument used to characterise  the maps of a functorial weak factorisation system in~\cite[\S2.4]{RosickyJ:laxfa} to establish Lemma~\ref{thm:char}.

\begin{lemma} \label{thm:char} Let $f : \Phi \rightarrow \Psi$ be a context morphism.
\begin{enumerate}[(i)]
\item   \label{itm:inj}  $f \in \mapA$ if and only if  $f$ is a type-theoretic injective equivalence.
\item  $f \in \mapB$ if and only if $f$ is a type-theoretic normal isofibration.  \label{itm:fib}
\end{enumerate}
\end{lemma}

\begin{proof} Recall that by Lemma~\ref{thm:factone} every map $f : \Phi \rightarrow \Psi $ admits a factorisation
\begin{equation*}
\xymatrix{
\Phi  \ar[r]^(.43){i_f}  &  \Id(f)  \ar[r]^{p_f} & \Psi \, ,}
\end{equation*}
where $i_f \in \mathcal{A}$ and $p_f \in \mathcal{B}$. Let us prove $(\ref{itm:inj})$. Define  $\mapA'$ to be the class of maps $f : \Phi \rightarrow \Psi$ such that the commutative diagram
\begin{equation}
\label{equ:fill}
\vcenter{\hbox{\xymatrix{
\Phi \ar[r]^(.45){i_f} \ar[d]_f  & \Id(f) \ar[d]^{p_f}  \\
\Psi \ar@{=}[r] & \Psi }}}
\end{equation}
has a diagonal filler. We claim that $\mapA = \mapA'$. To show $\mapA \subseteq \mapA'$, let $f : \Phi \rightarrow \Psi$ be in $\mapA$. The diagram in~(\ref{equ:fill}) has a diagonal filler since $f \in \mapA$ and $p_f \in \mapB$. To show $\mapA' \subseteq \mapA$, let $f : \Phi \rightarrow \Psi$ be in $\mapA'$, and assume to have a diagonal filler $j : \Psi \rightarrow \Id(f)$ for the diagram in~(\ref{equ:fill}). We can then exhibit $f$ as a retract of $i_f$ by the diagram
\[
\xymatrix{
\Phi \ar@{=}[r] \ar[d]_f & \Phi \ar@{=}[r] \ar[d]^{i_f} & \Phi \ar[d]^{f} \\
\Psi \ar[r]_(.4)j & \Id(f) \ar[r]_(.6){p_f} & \Psi }
\]
We have that $i_f \in \mapA$. Since the class $\mapA$, being defined by a weak orthogonality condition,  is closed under retracts, we have $f \in \mapA$. To conclude the proof, it is sufficient to
observe that $f$ is an injective equivalence if and only if $f \in \mapA'$. This involves unfolding the definitions of the context $\Id(f)$ and of  the morphism $i_f : \Phi \rightarrow \Id(f)$. For the proof of $(\ref{itm:fib})$, let  $\mapB'$ be the class of maps $f : \Phi \rightarrow \Psi$ such that the commutative diagram
\begin{equation*}
\xymatrix{
\Phi \ar[d]_(.45){i_f} \ar@{=}[r]  & \Phi \ar[d]^{f}  \\
\Id(f)  \ar[r]_(.53){p_f} & \Psi }
\end{equation*}
has a diagonal filler. The rest of argument follows along the lines of the one used to establish $(\ref{itm:inj})$ and hence
we omit it.
\end{proof}

We give two applications of Lemma~\ref{thm:char}. To motivate our first application,
let us recall that, given a weak factorisation system $(\mapA, \mapB)$ on a category $\catE$ with
finite products, it is possible to define a homotopy relation between maps of $\catE$ by
thinking of $\mapA$ and $\mapB$ as if they were the classes of acyclic
cofibrations and of fibrations of a Quillen model structure on $\catE$. For
convenience, let us assume that the weak factorisation system comes together
with an explicitly chosen factorisation of diagonal maps $\Delta_X : X
\rightarrow X \times X$ as an $\mapA$-map $i_X :  X \rightarrow X^I$ followed
by a $\mapB$-map $p_X : X^I \rightarrow X \times X$. We may then define two
maps $f : X \rightarrow Y$ and $g : X \rightarrow Y$ to be \myemph{homotopic}
if there exists a map $h : X \rightarrow Y^I$ such that the diagram
\[
\xymatrix{
  & X \ar[d]^h \ar@/^1pc/[ddr]^g \ar@/_1pc/[ddl]_f &  \\
  & \; Y^I  \ar[d]^{p_Y}  &  \\
 Y & Y \times Y \ar[l]^(.55){\pi_1}  \ar[r]_(.6){\pi_2} & Y}
  \]
 commutes. In the case of the identity type weak factorisation system, an
 explicit choice of factorisations of diagonal maps is given by
\[
\xymatrix{
\Phi  \ar[r]^(.4){i}  & \Id(\Phi)  \ar[r]^-{p} & (x \in \Phi, y \in \Phi)  \, }
\]
where the maps~$i$ and $p$ are given by $i \defeq (x, x, 1_x)$ and $p
\defeq  (x, y)$, respectively. It follows that two maps $f, g \colon \Phi \to
\Psi$ in the classifying category $\classT$ are homotopic if and only if we can derive a
judgement $( x \in \Phi) \  p(x) \in \Id(f(x), g(x))$ expressing their
pointwise propositional equality.
 The homotopy relation induced by a weak factorisation system $(\mapA, \mapB)$
 can always be shown to be reflexive and symmetric. However, without further assumptions on the weak factorisation
 system
we may not be able to show that it is transitive. One assumption that ensures
transitivity is that pullbacks of $\mapA$-maps in along $\mapB$-maps are again
$\mapA$-maps. In our example, the classifying category $\classT$ does not have
arbitrary pullbacks, but it does admit pullbacks along dependent
projections~\cite[\S6]{PittsA:catl}, which are $\mapB$-maps. The 
closure property of $\mapA$-maps stated in Proposition~\ref{thm:closure} below,
however,  suffices to show that the homotopy 
relation for the identity type weak factorisation system is transitive,
as an easy calculation shows.

\begin{proposition} \label{thm:closure}
Pullbacks of $\mapA$-maps along dependent projections are $\mapA$-maps.
\end{proposition}

\begin{proof} By Lemma~\ref{thm:char}, the claim follows once we show that, for a pullback diagram of form
\[
\xymatrix{
(x \in \Phi, z \in \Omega(f(x))) \ar[r] \ar[d]_{g} &   (x \in \Phi) \ar[d]^{f} \\
(y \in \Psi, z \in \Omega(y)) \ar[r]  & (y \in \Psi) }
\]
where $g = (f(x), z)$, if $f$ is an injective equivalence, then so is $g$. Since $f$ is an injective equivalence, we may assume to have
\[
(y \in \Psi) \ s(y) \in \Phi
\]
and the judgements in~(\ref{equ:injeqv1}), (\ref{equ:injeqv2}), (\ref{equ:injeqv3}).
Our first step in showing that $g$ is an injective equivalence will be to construct a judgement
\begin{equation*}
    (y \in \Psi, z \in \Omega(y))\ t(y, z) \in (x \in \Phi,\, z \in \Omega(f(x))
\end{equation*}
satisfying $    (x \in \Phi,\, z \in  \Omega(fx))\ tg(x, z) = (x, z) \in (x \in \Phi,   z \in \Omega(x))$. Now, to give
$t$ is equivalently to give judgements
\begin{gather*}
    (y \in \Psi,\, z \in \Omega(y)) \ t_1(y, z)  \in \Phi \, , \\
    (y \in \Psi,\, z \in \Omega(y)) \ t_2(y, z) \in \Omega(f(t_1(y, z)))\text.
\end{gather*}
So we define $t_1(y, z) \defeq s(y)$. Now must give an element $t_2(y, z) \in
\Omega(fsy)$. We obtain this by substituting $z \in \Omega(y)$ along
$\varepsilon_{y}^{-1} \in \Id(y,fsy)$ using the Leibniz rule:
\begin{equation*}
 ( y \in \Psi,\, z \in \Omega(y)) \  \, t_2(y, z) \defeq (\varepsilon_{y}^{-1})_!(z) \in \Omega(f(t_1(y, z))) \, .
\end{equation*}
Observe that we have
\begin{eqnarray*}
tg(x, z) &  = & t(fx, z) \\
              &  =  & (t_1(fx, z), t_2(fx, z))  \\
               & =  & (sfx, (\varepsilon_{fx}^{-1})_!(z)) \\
               & = & (x, (1_{fx})_!(z)) \\
               & = & (x, z)
\end{eqnarray*}
as required. We now come to the second step in the proof, which is to construct a judgement
\[
    (y \in \Psi,\, z \in \Omega(y)) \ \delta_{(y, z)} \in  \Id(gt(y, z), (y, z))
\]
satisfying $(x \in  \Phi,\, z \in \Omega (fx))\ \delta_{(fx, z)} = 1_{(fx, z)} \in \Id( (fx, z), (fx, z))$. Now, to give $\delta$ is the same as to give a judgement
\[
(y \in \Psi, \, z \in \Omega(y))\ \delta(y, z)  \in \Id\big(   (fsy,\, (\varepsilon_{y}^{-1})_!(z))  , (y, z) \big) \, .
\]
By the description of identity context in Section~\ref{sec:identitytypes}, to give this is equally well to give
a pair of judgements
\begin{gather*}
    (y \in \Psi,\, z \in \Omega(y)) \ \delta_1(y, z) \in \Id(fsy, y) \\
    (y \in \Psi,\, z \in  \Omega(y)) \ \delta_2(y, z) \in \Id \big( \delta_1(y, z)_!(\varepsilon_{y}^{-1})_!(z), z \big)
\end{gather*}
So we define $\delta_1(y, z) \defeq \varepsilon(y)$. We must now give an element
\[
\delta_2(y, z) \in \Id( (\varepsilon_{y})_{!}(\varepsilon_{y}^{-1})_!(z), z) \, .
\]
For this, let us show that we can derive a rule of the form

\[
\begin{prooftree}
p \in \Id(a,b) \qquad
(x \in \Phi) \ \Omega(x) \in \contype \qquad
e \in \Omega(b)
\justifies
\gamma_p(e) \in \Id( (p)_{!} (p^{-1})_{!}(e), e)
\end{prooftree}
\]
such that

\[
\begin{prooftree}
a \in \Phi  \quad
e \in \Omega(a)
\justifies
\gamma_{1_a}(e) = 1_e \in \Id(e,e)
\end{prooftree}
\medskip
\]
 By  elimination on $p \in \Id(a,b)$, we define
\[
\gamma_p(e) \defeq \myJ(a,b,p, [x] 1_{e} ) \in \Id( (p)_{!} (p^{-1})_{!}(e), e)
\]
Indeed, Lemma~\ref{thm:trans} implies that for $x \in \Phi$ and $z \in \Omega(x)$, we have
\begin{eqnarray*}
(1_x)_!   (1_{x}^{-1})_{!}(z)  & = & (1_{x}^{-1})_!(z) \\
 & = & (1_x)_!(z) \\
 &= & z \, .
 \end{eqnarray*}
 We can then define $\delta_2(y, z) \defeq \gamma_{\varepsilon_{y}}(z)$. This specifies $\delta$, and we now calculate
\begin{eqnarray*}
\delta{(fx, z)}  & = & (\delta_1(fx, z), \delta_2(fx, z)) \\
                     & =  &  (\varepsilon_{fx}, \gamma_{\varepsilon_{fx}}(z)) \\
                      & =  & (1_{fx}, \gamma_{1_{fx}}(z)) \\
                      &  = & (1_{fx}, 1_z) \\
                      & = & 1_{(fx, z)}
\end{eqnarray*}
as required.
\end{proof}

For our second application of Lemma~\ref{thm:char}, let us recall that
the category $\Gpd$ of small groupoids and functors admits a Quillen model structure $(\mathcal{W}, \mathcal{C},
\mathcal{F})$, in which the class of weak equivalences $\mathcal{W}$ consists of the categorical equivalences,
the class of fibrations $\mathcal{F}$ consists of the
Grothendieck fibrations, and the class of cofibrations $\mathcal{C}$ consists of the functors that are injective
on objects~\cite{AndersonD:fibgr,JoyalA:strscs}. As a consequence of this, both 
$(\mathcal{W} \cap \mathcal{C}, \mathcal{F})$ and $(\mathcal{C}, \mathcal{W}
\cap \mathcal{F})$ are weak factorisation systems on $\Gpd$. We shall be interested in relating the weak factorisation system $(\mathcal{W} \cap \mathcal{C}, \mathcal{F})$ on~$\Gpd$ with the identity type weak factorisation system on $\classT$.
Let us recall that a functor $f : A \rightarrow B$ between
groupoids is a Grothendieck fibration if and only if for every $\beta : f(a) \rightarrow b$ in $B$ there exists
$\alpha : a \rightarrow a'$ in $A$ such that $f(a') = b$ and~$f(\alpha) = \beta$. The required factorisation of a
functor $f : A \rightarrow B$ as an equivalence injective on objects followed by a Grothendieck fibration can
be obtained using the familiar \myemph{mapping space} construction,
\[
\xymatrix{
A \ar[r]^(.37){i_f} & \Path(f) \ar[r]^(.63){p_f} & B \, ,}
\]
where $\Path(f)$ is the groupoid whose objects consist of triples $(a,b,\beta)$, where
$a \in A$, $b \in B$, and $\beta : f(a) \rightarrow b$ in $B$.

\begin{theorem} Let  $f : \Phi \rightarrow \Psi$ be a context morphism.
\begin{enumerate}[(i)]
\item If $f \in \mapA$, then $\mathcal{F}(\Phi) \rightarrow \mathcal{F}(\Psi)$ is an equivalence injective on objects.
\item If $f \in \mapB$, then $\mathcal{F}(\Phi) \rightarrow \mathcal{F}(\Psi)$ is a Grothendieck fibration.
\end{enumerate}
\end{theorem}

\begin{proof} For part $(i)$, let $f : \Phi \rightarrow \Psi$ be in $\mathcal{A}$. By Lemma~\ref{thm:char}
$f$ is a type-theoretic
injective equivalence, so let us assume
\[
(y \in \Psi) \ s(y) \in \Phi
\]
and the judgements in~(\ref{equ:injeqv1}), (\ref{equ:injeqv2}),~(\ref{equ:injeqv3}). We  show that
$\mathcal{F}(s) : \mathcal{F}(\Psi) \rightarrow \mathcal{F}(\Phi)$ provides a quasi-inverse to
$\mathcal{F}(f) : \mathcal{F}(\Phi) \rightarrow \mathcal{F}(\Psi)$. First of all, we have a natural isomorphism
$\mathcal{F}(s) \circ \mathcal{F}(f) \Rightarrow 1_{\mathcal{F}(\Psi)}$ with components given by the maps $[\varepsilon_b] : f(s(b)) \rightarrow b$. To establish naturality, we need to show that for every $q \in
\Id(b_0, b_1)$, there is a homotopy between $q \circ \varepsilon_{b_0}$ and $\varepsilon_{b_1} \circ fs(q)$,
which can be proved by elimination on $q \in \Id(b_0, b_1)$. Secondly, we have
\[
\mathcal{F}(s) \circ \mathcal{F}(f)  = \mathcal{F}(s \circ f) = \mathcal{F}(1_\Phi) = 1_{\mathcal{F}(\Phi)} \, .
\]
which also shows that $\mathcal{F}(f)$ is injective on objects, as required. For part $(ii)$, let $f : \Phi \rightarrow
\Psi$ be a type-theoretic normal isofibration, and assume to have
\[
(x \in \Phi, y \in \Psi, u \in \Id(fx, y) ) \ j(x,y,u) \in \Phi
\]
and judgements as in~(\ref{equ:fib1}) and~(\ref{equ:fib2}).
By~(\ref{equ:fib1}), the map $j : \Id(f) \rightarrow \Phi$ makes the following diagram commute
\begin{equation}
\label{equ:last}
\vcenter{\hbox{\xymatrix{
 \Id(f) \ar[r]^(.55){j} \ar[dr]_{p_f}  & \Phi \ar[d]^{f} \\
                   & \Psi }}}
\end{equation}
To prove that $\mathcal{F}(f) : \mathcal{F}(\Phi) \rightarrow
\mathcal{F}(\Psi)$ is a Grothendieck fibration, we will need an auxiliary
result. We claim that there is a term
\begin{equation}
\label{equ:idfterm}
\big(x \in \Phi, y \in \Psi, u \in \Id(fx, y)\big) \ m(x, y, u) \in  \Id_{\Id(f)} \big( (x,f x, 1_{fx}), (x,y,u) \big)\text.
\end{equation}
To see this, let us write $\Omega(x, y, u) \defeq \Id_{\Id(f)} \big( (x,f x,
1_{fx}), (x,y,u) \big)$. Then we have a commutative square
\[
\xymatrix{
\Phi  \ar[r]^-k \ar[d]_{i_f} & (\Id(f), \Omega) \ar[d]^p \\
\Id(f) \ar@{=}[r]  & \Id(f) }
\]
where $k$ is given by $k \defeq (x, fx, 1_{fx}, 1_{(x, fx, 1_fx)})$, and $p$ is
the evident dependent projection. Because $i_f \in \mapA$ and $p \in \mapB$, we
have by Weak Orthogonality a diagonal filler $m \colon \Id(f) \to (\Id(f),
\Omega)$, as required for~(\ref{equ:idfterm}).

We now show that $\mathcal{F}(f)$ is a Grothendieck fibration. Let us consider
a map $\beta : f(a) \rightarrow b$ in $\mathcal{F}(\Psi)$. Let $p \in \Id(f(a),
b)$  such that $\beta = [p]$. Note that such a $p$ exists, but it is neither
unique nor determined canonically. We then define $a'  \defeq j(a,b,p)$. Next,
we need to define $\alpha : a \rightarrow a'$ in $\mathcal{F}(\Phi)$.
By~(\ref{equ:idfterm}), we have a term
\[
m(a, b, p) \in  \Id \big( (a,f a, 1_{fa}), (a,b,p) \big)\text.
\]
We define a map $\theta : (a, fa, 1_{fa}) \rightarrow (a, b, p)$ in
$\mathcal{F}(\Id(f))$ by letting $\theta = [m(a, b, p)]$. The required map
$\alpha : a \rightarrow a'$ can then be defined as the result of an application
of the functor $\mathcal{F}(j)$ to $\theta$. This has the required domain and
codomain, since $a = j(a,f(a), 1_{f(a)}) \in \Phi$ by (\ref{equ:fib2}), and $a'
= j(a,b,p) \in \Phi$ by the definition set earlier. Furthermore, the
commutativity of the diagram in~(\ref{equ:last}) implies that the result of
applying  $\mathcal{F}(f)$ to $\alpha$ is  $\beta$, as required.
\end{proof}

We can now compare the factorisations in $\classT$ and in $\Gpd$.

\begin{proposition} \label{thm:factgpd}
For every context morphism $f : \Phi \rightarrow \Psi$, we can define an equivalence
surjective on objects $\sigma_f : \mathcal{F}(\Id(f)) \rightarrow \Path(\mathcal{F}(f))$ making the following diagram commute
\[
\xymatrix{
\mathcal{F}(\Phi) \ar[rr]^{i_{\mathcal{F}(f)}} \ar[d]_{\mathcal{F}(i_f)} & &
\Path( \mathcal{F}(f)) \ar[d]^{p_{\mathcal{F}(f)}} \\
 \mathcal{F}(\Id(f)) \ar[urr] \ar[rr]_{\mathcal{F}(p_f)}  & & \mathcal{F}(\Psi) }
\]
\end{proposition}

\begin{proof} The objects of $\mathcal{F}(\Id(f))$ are triples $(a,b,p)$, where
$p \in \Id(f(a), b)$. The objects of $\Path(\mathcal{F}(f))$ are triples $a, b, \alpha$,
where $\alpha$ is an arrow $\alpha : f(a) \rightarrow b$ in~$\mathcal{F}(\Phi)$. Thus,
$\sigma_f$ can be defined as mapping $(a,b,p)$ to $(a, b, [p])$. Direct calculations
show the required properties.
\end{proof}

Let us write $\mathbf{J}$ for the groupoid with two objects and an isomorphism between them.
As a special case of Proposition~\ref{thm:factgpd}, we obtain that for every context~$\Phi$, there
is a surjective equivalence between
\[
\sigma : \mathcal{F}(x \in \Phi, y \in \Phi, u \in \Id(x,y))
\rightarrow \mathcal{F}(x \in \Phi)^\mathbf{J} \, ,
\]
where $\mathcal{F}(x \in \Phi)^\mathbf{J}$ can be seen as the groupoid of isomorphisms in
 $\mathcal{F}(x \in \Phi)$.

\section*{Acknowledgements} We are grateful to Michael Warren for suggesting to use display
maps rather than dependent projections in the definition of the identity type weak factorisation
system, thus simplifying our original definition.  We would also like to thank the referees for their useful comments.
The research described here was started during a visit of the first-named author to the Department of Mathematics of the University of Uppsala and discussed further during the workshop Categorical and Homotopical Structures in Proof Theory, held
at the Centre de Recerca Matem\`atica. We gratefully acknowledge the support of both institutions.  The second-named author is supported by a Research Fellowship of St John's College, Cambridge, and a Marie-Curie Intra-European Fellowship, Project no.~040802. The use of Paul Taylor's prooftree package is also
acknowledged.
\appendix

\section{Structural rules for dependent  type theories}
\label{app:rules}

\subsubsection*{Axiom, weakening and substitution}  The rule $(*)$ has the side-condition that the variable $x$ should not appear as a free variable in $\Gamma$ or $\Delta$.
When stating the rules below,~$\mathcal{J}$ stands for an arbitrary judgement.  \medskip

\[
\begin{prooftree}
\quad
\justifies
(\Gamma, x \in A) \ x \in A
\end{prooftree}
\]

\medskip

\[
\begin{array}{c}
\begin{prooftree}
(\Gamma, \Delta) \ \mathcal{J}   \quad
(\Gamma) \ A \in \type
\justifies
(\Gamma, x \in A, \Delta) \ \mathcal{J}
\using
(*)
\end{prooftree}  \qquad
\begin{prooftree}
(\Gamma, x \in A, \Delta) \ \mathcal{J}  \quad
(\Gamma) \ a \in A
\justifies
(\Gamma, \Delta[a/x] ) \ \mathcal{J}[a/x]
\end{prooftree}
\end{array}
\]

\medskip

\subsubsection*{Reflexivity, symmetry, and transitivity of definitional equality of types} \hfill

\[
\begin{array}{ccc}
\begin{prooftree}
 A \in \type
\justifies
 A = A
\end{prooftree} \qquad  &
\begin{prooftree}
 A  = B
\justifies
B = A
\end{prooftree} \qquad &
\begin{prooftree}
A  = B \quad
 B = C
\justifies
A = C
\end{prooftree}
\end{array}
\]

\medskip

\subsubsection*{Reflexivity, symmetry, and transitivity of definitional equality of objects}  \hfill

\[
\begin{array}{ccc}
\begin{prooftree}
 a \in A
\justifies
 a = a \in A
\end{prooftree} \qquad &
\begin{prooftree}
 a = b \in A
\justifies
b = a \in A
\end{prooftree} \qquad &
\begin{prooftree}
 a = b \in A \quad
b = c \in A
\justifies
a = c \in A
\end{prooftree}
\end{array}
\]

\medskip

\subsubsection*{Compatibility rules for definitional equality} \hfill

\[
\begin{array}{cc}
\begin{prooftree}
 a \in A \quad
 A = B
\justifies
 a \in B
\end{prooftree} \qquad &
\begin{prooftree}
 a  = b \in A \quad
A = B
\justifies
 a  = b \in B
\end{prooftree}
\end{array}
\]

\bibliographystyle{amsplain}
\bibliography{refs}

\providecommand{\bysame}{\leavevmode\hbox to3em{\hrulefill}\thinspace}
\providecommand{\MR}{\relax\ifhmode\unskip\space\fi MR }
\providecommand{\MRhref}[2]{%
  \href{http://www.ams.org/mathscinet-getitem?mr=#1}{#2}
}
\providecommand{\href}[2]{#2}
\begin{thebibliography}{10}

\bibitem{AndersonD:fibgr}
D.~W. Anderson, \emph{Fibrations and geometric realizations}, Bulletin of the
  American Mathematical Society \textbf{84} (1978), no.~5, 765--788.

\bibitem{AwodeyS:homtmi}
S.~Awodey and M.~A. Warren, \emph{Homotopy theoretic models of identity types},
  Mathematical Proceedings of the Cambridge Philosophical Society (To appear).

\bibitem{BlackwellR:twodmt}
R.~Blackwell, G.~M. Kelly, and A.~J. Power, \emph{Two-dimensional monad
  theory}, Journal of Pure and Applied Algebra \textbf{59} (1989), 1--41.

\bibitem{BousfieldA:confsc}
A.~K. Bousfield, \emph{Constructions of factorisation systems in categories},
  Journal of Pure and Applied Algebra \textbf{9} (1977), 207--220.

\bibitem{CartmellJ:genatc}
J.~Cartmell, \emph{Generalised algebraic theories and contextual categories},
  Annals of Pure and Applied Logic \textbf{32} (1986), 209--243.

\bibitem{deBruijnN:telmtl}
N.~G. de~Brujin, \emph{Telescopic mappings in typed lambda calculus},
  Information and Computation \textbf{91} (1991), no.~2, 189--204.

\bibitem{DybjerP:inttt}
P.~Dybjer, \emph{Internal type theory}, Types for proofs and programs
  (S.~Berardi and M.~Coppo, eds.), Lecture Notes in Computer Science, vol.
  1158, Springer, 1997, pp.~120--134.

\bibitem{GarnerR:twodmt}
R.~Garner, \emph{Two-dimensional models of type theory}, In preparation.

\bibitem{HofmannM:intttl}
M.~Hofmann, \emph{On the interpretation of type theory of locally cartesian
  closed categories}, Computer Science Logic 1994 (L.~Pacholski and J.~Tiuryn,
  eds.), Lecture Notes in Computer Science, vol. 933, Springer, 1995,
  pp.~427--441.

\bibitem{HofmannM:synsdt}
\bysame, \emph{Syntax and semantics of dependent type theories}, Semantics of
  Logics of Computation (P.~Dybjer and A.~M. Pitts, eds.), Cambridge University
  Press, 1997, pp.~79--130.

\bibitem{HofmannM:extcit}
\bysame, \emph{Extensional concepts in intensional type theory}, Springer,
  1999.

\bibitem{HofmannM:gromtt}
M.~Hofmann and T.~Streicher, \emph{The groupoid model of type theory},
  Twenty-five years of constructive type theory (G.~Sambin and J.~Smith, eds.),
  Oxford Logic Guides, vol.~36, Oxford University Press, 1998, pp.~83--111.

\bibitem{HoveyM:modc}
M.~Hovey, \emph{Model categories}, Mathematical Surveys and Monographs,
  vol.~63, American Mathematical Society, 1999.

\bibitem{HowardW:foratn}
W.~A. Howard, \emph{The formulae-as-types notion of construction}, To H. B.
  Curry: Essays in Combinatory Logic, Lambda Calculus, and Formalism (J.~R.
  Hindley and J.~P. Seldin, eds.), Academic Press, 1980, pp.~479--490.

\bibitem{JacobsB:catltt}
B.~Jacobs, \emph{Categorical logic and type theory}, Studies in Logic and the
  Foundations of Mathematics, vol. 141, Elsevier, 1999.

\bibitem{JoyalA:theqca}
A.~Joyal, \emph{The theory of quasi-categories and its applications},
  Quaderns~45, Centre de Recerca Matem\`atica, 2008.

\bibitem{JoyalA:strscs}
A.~Joyal and M.~Tierney, \emph{Strong stacks and classifying spaces}, Category
  Theory (Como 1990) (Berlin), Lecture Notes in Mathematics, vol. 1488,
  Springer, 1991, pp.~213--236.

\bibitem{LackS:homta2}
S.~Lack, \emph{Homotopy-theoretic aspects of 2-monads}, Journal of Homotopy and
  Related Structures \textbf{2} (2007), no.~2, 229--260.

\bibitem{LambekJ:inthoc}
J.~Lambek and P.~J. Scott, \emph{Introduction to higher-order categorical
  logic}, Cambridge Studies in Advanced Mathematics, vol.~7, Cambridge
  University Press, 1986.

\bibitem{LawvereW:funs}
F.~W. Lawvere, \emph{Functorial semantics}, Proc. Nat. Acad. Sci. USA
  \textbf{50} (1963), 869--872.

\bibitem{MaiettiM:modcbd}
M.~E. Maietti, \emph{Modular correspondence between dependent type theories and
  categorical universes including pretopoi and topoi}, Mathematical Structures
  in Computer Science \textbf{15} (2005), no.~6, 1089--1149.

\bibitem{MartinLofP:inttt}
P.~Martin-L\"of, \emph{Intuitionistic {T}ype {T}heory - {N}otes by {G}iovanni
  {S}ambin of a series of lectures given in {P}adua, {J}une 1980},
  Blibliopolis, 1984.

\bibitem{MayP:concat}
J.~P. May, \emph{A concise course in algebraic topology}, Chicago Lecture Notes
  in Mathematics, The University of Chicago Press, 1999.

\bibitem{NordstromB:promlt}
B.~Nordstr\"om, K.~Petersson, and J.~M. Smith, \emph{Programming in
  {M}artin-{L}\"of type theory: an introduction}, Oxford University Press,
  1990.

\bibitem{NordstromB:marltt}
\bysame, \emph{Martin-{L}\"of {T}ype {T}heory}, Handbook of Logic in Computer
  Science. Volume {V} (S.~Abramsky, D.~M. Gabbay, and T.~S.~E. Maibaum, eds.),
  Oxford University Press, 2001, pp.~1--37.

\bibitem{PittsA:catl}
A.~M. Pitts, \emph{Categorical logic}, Handbook of Logic in Computer Science.
  Volume {V} (S.~Abramsky, D.~M. Gabbay, and T.~S.~E. Maibaum, eds.), Oxford
  University Press, 2001, pp.~39--128.

\bibitem{RosickyJ:laxfa}
J.~Rosick\'y and W.~Tholen, \emph{Lax factorisation algebras}, Journal of Pure
  and Applied Algebra \textbf{175} (2002), no.~3, 335--382.

\bibitem{ScottD:conv}
D.~S. Scott, \emph{Constructive validity}, Symposium on Automated
  Demonstration, Lecture Notes in Mathematics, vol. 125, Springer, 1970,
  pp.~237--275.

\bibitem{ScottD:reltlc}
\bysame, \emph{Relating theories of the $\lambda$-calculus}, To H. B. Curry:
  Essays in Combinatory Logic, Lambda Calculus, and Formalism (J.~R. Hindley
  and J.~P. Seldin, eds.), Academic Press, 1980, pp.~403--450.

\bibitem{SeelyR:locccc}
R.~A.~G. Seely, \emph{Locally cartesian closed categories and type theories},
  Mathematical Proceedings of the Cambridge Philosophical Society \textbf{95}
  (1984), 33--48.

\bibitem{StreetR:fibys}
R.~Street, \emph{Fibrations and {Y}oneda's lemma in a 2-category}, Category
  Seminar (Proc. Sem. Sydney 1972/73) (G.~M. Kelly and R.~H. Street, eds.),
  Lecture Notes in Mathematics, vol. 420, Springer, 1974, pp.~104--133.

\bibitem{StreicherT:inviit}
T.~Streicher, \emph{Investigations into intensional type theory},
  Habilitationsschrift, Ludwig-Maximilians Universit\"at M\"unchen, 1993.

\bibitem{TaylorP:prafm}
P.~Taylor, \emph{Practical foundations of mathematics}, Cambridge Studies in
  Advanced Mathematics, vol.~59, Cambridge University Press, 1999.

\end{thebibliography}

\end{document}